\documentclass{amsart}

\usepackage[utf8]{inputenc}
\usepackage{amsmath, amssymb, amsthm}
\usepackage{mathtools}
\usepackage{geometry}

\usepackage{thmtools}
\usepackage{hyperref}

\usepackage{enumitem}
\usepackage{xcolor}
\usepackage{ytableau}
\usepackage{tikz}

\usepackage{todonotes}

\usepackage{cleveref}

\newcommand{\NN}{\mathbb{N}}
\newcommand{\ZZ}{\mathbb{Z}}

\newcommand{\tm}{\text{-}} 

\DeclareMathOperator{\DES}{DES}
\DeclareMathOperator{\maj}{maj}
\DeclareMathOperator{\NEG}{NEG}

\DeclareMathOperator{\pr}{pr}

\DeclareMathOperator{\cyc}{cyc}
\DeclareMathOperator{\col}{col}
\DeclareMathOperator{\row}{row}

\newcommand{\SYT}{\mathrm{SYT}}
\newcommand{\BST}{\mathrm{BST}}

\newcommand{\Dfn}[1]{\textcolor{blue}{\textbf{\emph{#1}}}} 

\newcommand{\qbinom}[3][q]{\begin{bmatrix} #2 \\ #3 \end{bmatrix}_{#1}}

\newtheorem{theorem}{Theorem}[section]
\newtheorem{lemma}[theorem]{Lemma}
\newtheorem{proposition}[theorem]{Proposition}
\newtheorem{corollary}[theorem]{Corollary}

\theoremstyle{definition}
\newtheorem{definition}[theorem]{Definition}
\newtheorem{example}[theorem]{Example}

\theoremstyle{remark}
\newtheorem{remark}[theorem]{Remark}

\Crefname{example}{Example}{Examples}
\crefname{example}{example}{examples}
\Crefname{proposition}{Proposition}{Propositions}
\crefname{proposition}{proposition}{propositions}

\title{Super Major Index Cyclic Sieving}
\author{Stephan Pfannerer}
\address{\parbox{\linewidth}{Dept.\ of Combinatorics \& Optimization, University of Waterloo, Waterloo, ON, N2L 3G1, Canada}}
\email{\parbox[t]{\linewidth}{math@pfannerer-mittas.net}}
\date{\today}

\begin{document}

\maketitle

\begin{abstract}
Recently, Armon and Swanson introduced signed standard tableaux and a corresponding super major index that refines the classical major index. In this paper, we prove that signed standard tableaux of rectangular shape exhibit a cyclic sieving phenomenon (CSP) under the combined action of Schützenberger promotion and cyclic shift of the signs, with the sieving polynomial given by the super major index generating function. This extends Rhoades's celebrated CSP for standard Young tableaux. Furthermore, by considering Cartesian products of tableaux, we generalize this result to arbitrary non-rectangular shapes.
\end{abstract}

\section{Introduction}

The \Dfn{cyclic sieving phenomenon} (CSP), introduced by Reiner, Stanton, and White \cite{RSW2004}, reveals a surprising connection between the enumeration of symmetry classes and the evaluation of $q$-analogues at roots of unity. This phenomenon generalizes Stembridge's earlier ``$q=-1$ phenomenon'' \cite{Stembridge1994} and has since been found in diverse combinatorial contexts (see \cite{Sagan2011} for a survey).

A celebrated instance of the CSP involves \Dfn{standard Young tableaux} (SYT). For a rectangular partition $\lambda$, Rhoades \cite{Rhoades2010} proved that the set $\SYT(\lambda)$ exhibits a CSP under the action of Schützenberger's \Dfn{promotion}, with the sieving polynomial given by the $q$-analogue of the Frame--Robinson--Thrall hook length formula \cite{FrameRobinsonThrall1954}, which agrees up to a simple twist with the \Dfn{major index} generating series. While promotion is well-behaved on rectangles, its order and structure on arbitrary shapes are more complex. In recent joint work with Alexandersson et al. \cite{APRU2021} we established an ``implicit'' CSP for general shapes with the same sieving polynomial, though the explicit bijective action remains conjectural or completely unknown in many cases.

In this paper, we extend the study of cyclic sieving to 
\Dfn{signed standard tableaux}\footnote{Armon and Swanson call them \Dfn{standard super tableaux}, but that term is ambiguous in literature}, objects recently investigated by Armon and Swanson \cite{ArmonSwanson2025} in the context of Lie superalgebras $\mathfrak {gl}(m|n)$.
A signed standard tableau extends the classical definition by distinguishing a subset of ``negative'' entries. They introduced a \Dfn{super major index} statistic on these tableaux, which refines the classical major index and encodes the interplay between descents and negative entries. Our main result establishes that rectangular signed standard tableaux together with the super major index generating series exhibit the cyclic sieving phenomenon under a super-analogue of promotion.

\begin{theorem}[Main Theorem A] \label{thm:mainA}
    Let $\lambda \vdash n$ be a rectangular partition. For any $0 \le k \le n$, the triple
    \[
        \left(\SYT_{\pm k}(\lambda), \langle \pr \rangle, q^{\gamma(n,k)-\kappa(\lambda)} f^\lambda_{\pm k}(q) \right)
    \]
    exhibits the cyclic sieving phenomenon, where the action is Schützenberger promotion on the tableau and cyclic shift on the negative set. The sieving polynomial $f^\lambda_{\pm k}(q)$ is the generating function for the super major index over signed tableaux with $k$ negative entries, and the twist $\kappa(\lambda)$ is an integer depending on the shape of $\lambda$ and $\gamma(n,k)= (n-1)\binom{k}{2}$.
\end{theorem}

We also consider the case of arbitrary shapes. While a CSP may not exist for a single copy of $\SYT(\lambda)$ due to negative character evaluations, \cite{APRU2021} have shown that for any shape $\lambda$, a CSP exists for the $m$-fold Cartesian product $\SYT(\lambda)^m$ provided the polynomial evaluations are non-negative (which is always true if $m$ is even). We also extend this result to the signed case.

\begin{theorem}[Main Theorem B]\label{thm:mainB}
    Let $\lambda$ be a partition of $n$. For any $m \ge 1$, there exists a cyclic group action of order $n$ on $\SYT_{\pm k}(\lambda)^m$ such that the triple
    \[
    (\SYT_{\pm k}(\lambda)^m, C_n, (q^{\gamma(n,k)} f^\lambda_{\pm k}(q))^m)
    \]
    exhibits the cyclic sieving phenomenon  if and only if for all $d \mid n$, the evaluation $(f^\lambda(\xi^d))^m$ is a non-negative integer, where $\xi$ is a primitive $n$-th root of unity. In particular, if $m$ is even, such a CSP always exists.
\end{theorem}

The paper is organized as follows. In \Cref{sec:background}, we review the necessary background on standard Young tableaux, the cyclic sieving phenomenon, and the connection between evaluations of major index polynomials at roots of unity and the enumeration of border strip tableaux. In \Cref{sec:super_tableaux}, we introduce signed standard tableaux and the super major index, and we present a key factorization of its generating function. \Cref{sec:proof} contains the proof of our main results; we state and prove a unified theorem and show how both main theorems follow as special cases.

\section{Background}\label{sec:background}

\subsection{Standard Young tableaux}
A \Dfn{partition} of a positive integer $n$ is a sequence 
\[
\lambda = (\lambda_1 \ge \lambda_2 \ge \cdots \ge \lambda_\ell > 0)
\]
of positive integers whose sum is $n$.  The Young diagram of shape $\lambda$ is a left-justified array of $\lambda_i$ boxes in row $i$.
If the Young diagram of $\lambda$ is a rectangle, i.e., $\lambda = (\underbrace{b,b,\dots, b}_{\text{$a$ copies}})$ for some $a,b \ge 1$, we say $\lambda$ is a \Dfn{rectangular partition} and write $\lambda=a\times b$.

If $\lambda$ and $\mu$ are partitions, we say $\lambda$ \Dfn{contains} $\mu$ if the Young diagram of $\mu$ is contained in that of $\lambda$. In this case, the \Dfn{skew shape} $\lambda/\mu$ is the set-theoretic difference of the two diagrams.

A \Dfn{standard Young tableau} (SYT) of shape $\lambda$ is a filling of the boxes of the Young diagram with the numbers $1,2,\dots,n$, each used exactly once, such that the entries strictly increase left-to-right in each row and strictly increase top-to-bottom in each column.  We denote the set of standard Young tableaux of shape $\lambda$ by $\SYT(\lambda)$.

Given a tableau $T \in \SYT(\lambda)$, an integer $i \in \{1,2,\dots,n-1\}$ is called a \Dfn{descent} of $T$ if the entry $i+1$ appears in a strictly lower row than the entry $i$.  The \Dfn{descent set} of $T$ is
\[
\DES(T) = \{\, i \in \{1,\dots,n-1\} : i+1 \text{ lies in a lower row than } i \,\}.
\]
The \Dfn{major index} of $T$ is the sum of its descents:
\[
\maj(T) = \sum_{i \in \DES(T)} i.
\]

We denote the generating function for the major index by
\[
f^\lambda(q) \coloneqq \sum_{T \in \SYT(\lambda)} q^{\maj(T)}.
\]

\begin{example}\label{ex:maj}
    Consider the following standard Young tableau:
    \[
    T = \ytableaushort{124,35,67}
    \]
    The descent set is $\DES(T) = \{2, 4, 5\}$. The major index is $\maj(T) = 2 + 4 + 5 = 11$.
\end{example}

For a cell $\square$ in the Young diagram of $\lambda$ we denote its row and column by $\row(\square)$ and $\col(\square)$, respectively, where the top-left cell is at row $1$ and column $1$. For a cell $\square \in \lambda$ its \Dfn{content} is defined as $c(\square) = \col(\square) - \row(\square)$.
The \Dfn{hook} of $\square$ consists of the cell $\square$ itself, all cells to its right in the same row, and all cells below it in the same column. The number of cells in the hook is the \Dfn{hook length}, denoted $h_\square$.

\begin{example}
    Let $\lambda = (6,5,4,2,2,2)$. The contents and hook lengths of the cells are:
    \[
\ytableaushort{012345,{\tm1}0123,{\tm2}{\tm1}01,{\tm3}{\tm2},{\tm4}{\tm3},{\tm5}{\tm4}}\qquad\qquad
\ytableaushort{{11}{10}6531,98431,7621,43,32,21}\]
\end{example}

The total number of standard Young tableaux of a given shape is given by the celebrated \Dfn{hook-length formula} of Frame, Robinson, and Thrall \cite{FrameRobinsonThrall1954}.
\begin{theorem}[Hook-Length Formula]
    Let $\lambda$ be a partition of $n$. The number of standard Young tableaux of shape $\lambda$ is
    \[
        |\SYT(\lambda)| = \frac{n!}{\prod_{\square \in \lambda} h_\square},
    \]
    where the product is over all cells $u$ in the Young diagram of $\lambda$.
\end{theorem}

\subsection{Promotion}
Schützenberger promotion is a bijection on standard Young tableaux (SYT) of a fixed shape.  
Although we do not require the description of promotion itself for the statement of our main results, we recall its definition here for completeness, as it is the group action underlying our CSP.

Given a tableau $T \in \SYT(\lambda)$ of size $n$, its promotion $\pr(T)$ is defined as follows:

\begin{enumerate}
    \item Remove the entry $1$, creating an empty cell.
    \item Perform \Dfn{jeu-de-taquin slides}: at each step slide into the empty cell the smaller of the entries immediately to its right or below, until the empty cell reaches an outer corner.
    \item Insert $n+1$ into the empty cell, yielding a tableau with entries $\{2,3,\dots,n+1\}$.
    \item Subtract $1$ from every entry.
\end{enumerate}

The resulting tableau is the promoted tableau $\pr(T)$.  
Promotion is invertible, preserves shape, and acts with finite order; for rectangular shapes $a \times b$ it has order $ab$. See e.g. \cite[Theorem 1.1]{Rhoades2010}.

\subsection{$q$-analogues}
We recall standard $q$-analogues. For a positive integer $n$, the \Dfn{$q$-integer} is defined as
\[
    [n]_q \coloneqq 1 + q + \dots + q^{n-1} = \frac{1-q^n}{1-q},
\]
with $[0]_q = 0$. The \Dfn{$q$-factorial} is defined by $[n]_q! \coloneqq \prod_{i=1}^n [i]_q$, with $[0]_q! = 1$.
The \Dfn{$q$-binomial coefficient} is given by
\[
    \qbinom{n}{k} \coloneqq \frac{[n]_q!}{[k]_q! [n-k]_q!}
\]
for $0 \le k \le n$, and $0$ otherwise.

The generating function for the major index over $\SYT(\lambda)$ is given by a $q$-analogue of the hook-length formula~\cite[Corollary 7.21.5]{Stanley1999}.
\begin{theorem}[$q$-Hook-Length Formula]
    Let $\lambda$ be a partition of $n$. The major index generating function is given by
    \begin{equation}\label{eq:q hook length formula}
        f^\lambda(q) = q^{\kappa(\lambda)} \frac{[n]_q!}{\prod_{\square \in \lambda} [h_{\square}]_q},
    \end{equation}
    where $\kappa(\lambda) = \sum_{i \ge 1} (i-1)\lambda_i$.
\end{theorem}

\subsection{The Cyclic Sieving Phenomenon}
Let $X$ be a finite set of combinatorial objects, let $C = \langle g \rangle$ be a cyclic group of order $n$ acting on $X$, and let $P(q) \in \NN[q,q^{-1}]$ be a Laurent polynomial\footnote{Note that classically, one considers polynomials and not Laurent polynomials. We allow the more general setup here, to avoid extra multiplications with appropriate powers of $q^n$ in our theorems.} with non-negative integer coefficients. Let $\xi$ be a complex primitive $n$-th root of unity.

\begin{definition}[\cite{RSW2004}]
    The triple $(X, C, P)$ is said to exhibit the \Dfn{cyclic sieving phenomenon} (CSP) if for all integers $d \ge 0$,
    \[
        |\{x \in X : g^d \cdot x = x\}| = P(\xi^d).
    \]
\end{definition}

We also write $X^g$ for the set of fixed points of $g$ on $X$ and call the triple $(X, C, P)$ a \Dfn{CSP triple} if it exhibits the cyclic sieving phenomenon.

Recall two fundamental CSPs for subsets and rectangular standard Young tableaux that we will utilize.
\begin{theorem}[\cite{RSW2004}]\label{thm:subsets}
    Let $0 \le k \le n$. The triple
    \[
        \left( \binom{[n]}{k}, \langle \cyc \rangle, \qbinom{n}{k} \right)
    \]
    exhibits the cyclic sieving phenomenon, where $\cyc$ acts on $k$-element subsets of $[n]$ by cyclic shift, i.e., decrementing each entry greater than $1$ by $1$ and changing $1$ to $n$.
    In particular, for a divisor $d \mid n$ and a primitive $n$-th root of unity $\xi$, let $s=n/d$. Then
    \begin{equation}\label{eq:q_binom_eval}
        \qbinom[q=\xi^d]{n}{k} = \begin{cases} \binom{d}{k/s} & \text{if } s \mid k, \\ 0 & \text{otherwise.} \end{cases}
    \end{equation}
\end{theorem}

\begin{theorem}[\cite{Rhoades2010}]\label{thm:Rhoades}
    Let $\lambda \vdash n$ be a rectangular partition. The triple
    \[
        \left( \SYT(\lambda), \langle \pr \rangle, q^{-\kappa(\lambda)}f^\lambda(q) \right)
    \]
    exhibits the cyclic sieving phenomenon, where $\pr$ acts by Schützenberger promotion.
\end{theorem}

Additionally, Rhoades' result extends to the case of tuples of standard Young tableaux of same arbitrary shape $\lambda$, but here the group action is unknown.
\begin{theorem}[\cite{APRU2021}]\label{thm:APRU}
    Let $\lambda \vdash n$ be any partition. There exists a cyclic group action of order $n$ on $\SYT(\lambda)$ such that the triple
    \[
    (\underbrace{\SYT(\lambda) \times \SYT(\lambda) \times \dots \times \SYT(\lambda)}_{\text{$m$ copies}} , C_n, (f^\lambda(q))^m)
    \]
    exhibits the cyclic sieving phenomenon if and only if for all $d \mid n$, $(f^\lambda(\xi^d))^m$ is a non-negative integer.
    
    In particular, if $m$ is even, such a CSP always exists.
\end{theorem}

We can construct new CSP triples from existing ones via the Cartesian product.

\begin{lemma}[Product CSP] \label{lem:product_csp}
    Let $(X, C, P)$ and $(Y, C, Q)$ be two triples exhibiting the cyclic sieving phenomenon with the same group $C = \langle g \rangle$. Then the triple
    \[
        (X \times Y, C, P\cdot Q)
    \]
    exhibits the CSP, where the group action on $X \times Y$ is defined diagonally by $g \cdot (x,y) = (g \cdot x, g \cdot y)$.
\end{lemma}

\begin{proof}
    Let $d \ge 0$. The set of fixed points of $g^d$ on the product is the Cartesian product of the fixed point sets. That is
    \[
    (X \times Y)^{g^d} = \{(x,y) : g^d x = x \text{ and } g^d y = y\} = X^{g^d} \times Y^{g^d}.
    \]
    Taking cardinalities, we have
    \[
    |(X \times Y)^{g^d}| = |X^{g^d}| \cdot |Y^{g^d}|.
    \]
    Since the individual triples exhibit the CSP, $|X^{g^d}| = P(\xi^d)$ and $|Y^{g^d}| = Q(\xi^d)$. Therefore,
    \[
    |(X \times Y)^{g^d}| = P(\xi^d) Q(\xi^d) = (P\cdot Q)(\xi^d).
    \]
    Thus, the product polynomial correctly enumerates the fixed points.
\end{proof}

\subsection{Border Strip Tableaux}
The evaluation of the standard $q$-hook length formula at roots of unity is closely related to the enumeration of border strip tableaux.

\begin{definition}
    A \Dfn{border strip} is a connected skew shape that does not contain a $2 \times 2$ square. A \Dfn{border strip tableau} of shape $\lambda$ and type $\mu = (\mu_1,\dots,\mu_k)$ is a filling of the Young diagram of $\lambda$ with the numbers $1,\dots, k$ such that
    \begin{enumerate}
        \item the entries are weakly increasing in each row from left to right, and
        \item the entries are weakly increasing in each column from top to bottom, and
        \item the cells filled with $i$ form a border strip of size $\mu_i$.
    \end{enumerate}
 We denote by $\BST(\lambda, k)$ the set of border strip tableaux of shape $\lambda$ such that all border strips have size $k$.
\end{definition}

\begin{example}
A border strip tableau of shape $(6,5,4,2,2,2)$ with strips of length $3$ is:
    \[
\begin{tikzpicture}[x=1.2em, y=1.2em,baseline={([yshift=-1ex]current bounding box.center)}]
\draw[line width=0.8] (0,0) -- (6,0) -- ++(0,-1) -- ++(-1,0) -- ++(0,-1) -- ++(-1,0) -- ++(0,-1) -- ++(-2,0) -- ++(0,-1) -- ++(-0,0) -- ++(0,-1) -- ++(-0,0) -- ++(0,-1) -- ++(-2,0) -- cycle;
\draw[line width=0.8] (0,0) -- (4,0) -- ++(0,-1) -- ++(-0,0) -- ++(0,-1) -- ++(-0,0) -- ++(0,-1) -- ++(-2,0) -- ++(0,-1) -- ++(-0,0) -- ++(0,-1) -- ++(-0,0) -- ++(0,-1) -- ++(-2,0);
\draw[line width=0.8] (0,0) -- (4,0) -- ++(0,-1) -- ++(-0,0) -- ++(0,-1) -- ++(-0,0) -- ++(0,-1) -- ++(-2,0) -- ++(0,-1) -- ++(-1,0) -- ++(0,-1) -- ++(-1,0);
\draw[line width=0.8] (0,0) -- (4,0) -- ++(0,-1) -- ++(-0,0) -- ++(0,-1) -- ++(-0,0) -- ++(0,-1) -- ++(-4,0);
\draw[line width=0.8] (0,0) -- (3,0) -- ++(0,-1) -- ++(-0,0) -- ++(0,-1) -- ++(-0,0) -- ++(0,-1) -- ++(-3,0);
\draw[line width=0.8] (0,0) -- (3,0) -- ++(0,-1) -- ++(-0,0) -- ++(0,-1) -- ++(-3,0);
\draw[line width=0.8] (0,0) -- (2,0) -- ++(0,-1) -- ++(-1,0) -- ++(0,-1) -- ++(-1,0);
\draw[line width=0.8] (0,0) -- (0,0);

\draw (4.5,-1.5) node{7};\draw (4.5,-0.5) node{7};\draw (5.5,-0.5) node{7};
\draw (0.5,-5.5) node{6};\draw (1.5,-5.5) node{6};\draw (1.5,-4.5) node{6};
\draw (0.5,-4.5) node{5};\draw (0.5,-3.5) node{5};\draw (1.5,-3.5) node{5};
\draw (3.5,-2.5) node{4};\draw (3.5,-1.5) node{4};\draw (3.5,-0.5) node{4};
\draw (0.5,-2.5) node{3};\draw (1.5,-2.5) node{3};\draw (2.5,-2.5) node{3};
\draw (1.5,-1.5) node{2};\draw (2.5,-1.5) node{2};\draw (2.5,-0.5) node{2};
\draw (0.5,-1.5) node{1};\draw (0.5,-0.5) node{1};\draw (1.5,-0.5) node{1};
\end{tikzpicture}
\]
\end{example}

\begin{theorem}[{\cite[Theorem 2.7.27]{James1984}.}] \label{thm:james kerber}
    Let $\lambda \vdash n$ and let $\xi$ be a primitive $d$-th root of unity. If $d$ does not divide $n$, then $f^\lambda(\xi) = 0$. If $d \mid n$, then
    \begin{equation}\label{eq:MN Rule}
        f^\lambda(\xi) = \epsilon_\lambda |\BST(\lambda, d)|,
    \end{equation}
    where $\epsilon_\lambda \in \{-1, 1\}$ is a sign depending on the shape $\lambda$.
\end{theorem}

The crucial step in our proof of the main theorems relies on the behavior of cell contents for shapes $\lambda$ with $\BST(\lambda, m) \neq \emptyset$.

\begin{lemma} \label{lem:content_dist}
    Let $\lambda$ be a partition of $n$. Let $s$ be a divisor of $n$ such that $\BST(\lambda, s) \neq \emptyset$. Then the multiset of contents $\{\!\{c(\square) \pmod{s} : \square \in \lambda\}\!\}$ consists of the values $\{0, 1, \dots, s-1\}$, each appearing with multiplicity $n/s$.
\end{lemma}

\begin{proof}
    This follows from the structural properties of border strip decompositions. If a shape can be tiled by border strips of length $s$, the contents of the cells in any single border strip are distinct modulo $s$ and cover the full set of residues $\{0, \dots, s-1\}$ exactly once. Since there are $n/s$ strips in the decomposition, the total collection of contents covers each residue $n/s$ times.
\end{proof}

\section{Signed Standard Tableaux and Super Major Index} \label{sec:super_tableaux}

Let $\lambda \vdash n$ be a partition. A \Dfn{signed standard tableau} of shape $\lambda$ is a pair $\mathcal{T} = (\mathcal{T}^+, D)$, where:
\begin{itemize}
    \item $\mathcal{T}^+ \in \SYT(\lambda)$ is a standard Young tableau of shape $\lambda$, and
    \item $D \subseteq \{1, 2, \dots, n\}$ is a subset of entries, which we refer to as the \Dfn{negative entries} of $\mathcal{T}$, denoted by $\NEG(\mathcal{T})$.
\end{itemize}
The set of all signed standard tableaux of shape $\lambda$ is denoted by $\SYT_{\pm}(\lambda)$. The subset of tableaux with exactly $k$ negative entries is denoted by $\SYT_{\pm k}(\lambda)$.

We define a cyclic action $\pr$ on $\SYT_{\pm}(\lambda)$ by acting on the components:
\[
\pr(\mathcal{T}) \coloneqq (\pr(\mathcal{T}^+), \cyc(D)).
\]

\begin{definition}[{\cite{ArmonSwanson2025}}]
    An integer $i \in \{1, \dots, n-1\}$ is a \Dfn{super descent} of $\mathcal{T} \in \SYT_{\pm}(\lambda)$ if either:
    \begin{enumerate}
        \item $i \in \DES(\mathcal{T}^+)$ and $i+1 \notin \NEG(\mathcal{T})$, or
        \item $i \notin \DES(\mathcal{T}^+)$ and $i \in \NEG(\mathcal{T})$.
    \end{enumerate}
    The \Dfn{super major index} of $\mathcal{T}$, denoted $\maj(\mathcal{T})$, is the sum of its super descents.

    We denote with \[f^\lambda_{\pm k}(q) = \sum_{\mathcal{T} \in \SYT_{\pm k}(\lambda)} q^{\maj(\mathcal{T})}\] the generating function for the super major index over signed tableaux with $k$ negative entries and with \[f^\lambda(q,t) \coloneqq \sum_k f^\lambda_{\pm k}(q) t^k = \sum_{\mathcal{T}} q^{\maj(\mathcal{T})} t^{|\NEG(\mathcal{T})|}\] the bivariate generating function summing over all signed standard tableaux that also tracks the number of negative entries with $t$.
\end{definition}

\begin{example}
    Consider the tableau from the previous \Cref{ex:maj}.
    \[
    \mathcal{T}^+ = \ytableaushort{124,35,67}, \qquad \NEG(\mathcal{T}) = \{3, 6\}.
    \]
    We may visualize this by marking negative entries (e.g., with an overline):
    \[
    \mathcal{T} = \ytableaushort{124,{\overline{3}}5,{\overline{6}}7}
    \]
    The descents of $\mathcal{T}^+$ are $\DES(\mathcal{T}^+) = \{2, 4, 5\}$.
    Let us compute the super descents:
    \begin{itemize}
        \item $i=1$: $1 \notin \DES$, $1 \notin \NEG$. No.
        \item $i=2$: $2 \in \DES$, $3 \in \NEG$. No.
        \item $i=3$: $3 \notin \DES$, $3 \in \NEG$. (Condition 2 holds). \textbf{Yes}.
        \item $i=4$: $4 \in \DES$, $5 \notin \NEG$. (Condition 1 holds). \textbf{Yes}.
        \item $i=5$: $5 \in \DES$, $6 \in \NEG$. No.
        \item $i=6$: $6 \notin \DES$, $6 \in \NEG$. (Condition 2 holds). \textbf{Yes}.
    \end{itemize}
    Thus, the super descent set is $\{3, 4, 6\}$ and $\maj(\mathcal{T}) = 13$.
\end{example}

Before analyzing the super major index, we observe that \Cref{lem:product_csp} immediately provides a cyclic sieving phenomenon for super tableaux, viewing them simply as the Cartesian product $\SYT(\lambda) \times \binom{[n]}{k}$.

\begin{proposition}[Trivial Super CSP]\label{prop:trivial CSP}
    Let $\lambda \vdash n$ be a rectangular partition. The triple
    \[
    \left( \SYT_{\pm k}(\lambda), \langle \pr \times \cyc \rangle, q^{-\kappa(\lambda)} f^\lambda(q) \cdot \qbinom{n}{k} \right)
    \]
    exhibits the cyclic sieving phenomenon.
\end{proposition}

\begin{remark}
    While valid, this polynomial is distinct from the one in \Cref{thm:mainA}.

    For instance consider the partition $\lambda = (3,3)$ and $k=1$.
     The generating function for the classical major index is 
    \[
        f^{(3,3)}(q) = q^3 + q^5 + q^6 + q^7 + q^9.
    \]
    The polynomial for the ``Trivial Super CSP'' given by \Cref{prop:trivial CSP} is:
    \begin{align*}
        P_{\text{trivial}}(q) &= q^{-3} f^{(3,3)}(q) [6]_q \\
        &= 1 + q + 2q^2 + 3q^3 + 4q^4 + 4q^5 + 4q^6 + 4q^7 + 3q^8 + 2q^9 + q^{10} + q^{11}.
    \end{align*}

    For this example the super major index generating function is:
    \begin{align*}
        f^{(3,3)}_{\pm 1}(q) = q^2 + 2q^3 + 3q^4 + 4q^5 + 5q^6 + 5q^7 + 4q^8 + 3q^9 + 2q^{10} + q^{11}.
    \end{align*}
    Multiplying by the twist $q^{\gamma(6,1) - \kappa(3,3)} =  q^{-3}$ yields our main sieving polynomial:
    \[
        P_{\text{super}}(q) = q^{-1} + 2 + 3q + 4q^2 + 5q^3 + 5q^4 + 4q^5 + 3q^6 + 2q^7 + q^8.
    \]

\end{remark}

Armon and Swanson proved that the super major index generating function has a product formula that refines the $q$-hook length formula by incorporating the contents of the cells.
\begin{theorem}[{\cite[Theorem 1.6]{ArmonSwanson2025}}] \label{thm:super_major_index_hook_lenght}
    Let $\lambda$ be a partition of $n$. The super major index generating function is given by
    \begin{equation}  \label{eq:qt hook length formula}
    f^\lambda(q,t) = [n]_q! \prod_{\square \in \lambda} \frac{q^{\row(\square)-1} + t q^{\col(\square)-1}}{[h_{\square}]_q}.
    \end{equation}

\end{theorem}

From this we obtain that the generating function $f^\lambda(q,t) \coloneqq \sum_{\mathcal{T}} q^{\maj(\mathcal{T})} t^{|\NEG(\mathcal{T})|}$ factors in a way crucial for our analysis.

\begin{corollary} \label{cor:content_formula}
    The super major index generating function factors as:
    \begin{equation}\label{eq:factorisation}
    f^\lambda(q,t) = f^\lambda(q) \prod_{\square \in \lambda} (1 + t q^{c(\square)}),
    \end{equation}
    where $c(\square)$ is the content of the cell $\square$.
\end{corollary}

\begin{proof}
    We factor out $q^{\row(\square)-1}$ from the numerator of each term in the product in~\eqref{eq:qt hook length formula}:
    \begin{align*}
        q^{\row(\square)-1} + t q^{\col(\square)-1} &= q^{\row(\square)-1} (1 + t q^{\col(\square) - \row(\square)}) \\
        &= q^{\row(\square)-1} (1 + t q^{c(\square)}),
    \end{align*}
    where $c(\square) = \col(\square) - \row(\square)$ is the content of the cell $\square$.
    Substituting this back into the expression for $f^\lambda(q,t)$, we obtain:
    \begin{align*}
        f^\lambda(q,t) &= [n]_q! \prod_{\square \in \lambda} \frac{q^{\row(\square)-1} (1 + t q^{c(\square)})}{[h_{\square}]_q} \\
        &= \left( [n]_q! \frac{\prod_{\square \in \lambda} q^{\row(\square)-1}}{\prod_{\square \in \lambda} [h_{\square}]_q} \right) \prod_{\square \in \lambda} (1 + t q^{c(\square)}).
    \end{align*}
    The first factor can be simplified by observing that $\sum_{\square \in \lambda}(\row(\square)-1) = \sum_{i \ge 1} \lambda_i (i-1) = \kappa(\lambda)$. Thus,
    \[
        [n]_q! \frac{q^{\sum_{\square \in \lambda}(\row(\square)-1)}}{\prod_{\square \in \lambda} [h_{\square}]_q} = q^{\kappa(\lambda)} \frac{[n]_q!}{\prod_{\square \in \lambda} [h_{\square}]_q} = f^\lambda(q),
    \]
    by the $q$-Hook-Length Formula~\eqref{eq:q hook length formula}.
\end{proof}

\section{Proof of the Main Theorems} \label{sec:proof}

In this section, we first unify \Cref{thm:mainA,thm:mainB} and then proceed to prove the unified statement. The key idea is to leverage the product formula for the super major index generating function and analyze its evaluation at roots of unity in relation to border strip tableaux.

\begin{theorem}\label{thm:unified}
    Let $\lambda$ be a partition of $n$, $m \ge 1$, and $\alpha \in \ZZ$, such there exists a cyclic group action of order $n$ generated by $g$ on $\SYT(\lambda)^m$ such that the triple
    \[
    (\SYT(\lambda)^m, \langle g \rangle, q^{\alpha}(f^\lambda(q))^m)
    \]
    exhibits the cyclic sieving phenomenon.
    
    Then the triple
    \[
    (\SYT_{\pm k}(\lambda)^m, \langle g \times \cyc^m \rangle, q^{\alpha}(q^{\gamma(n,k)} f^\lambda_{\pm k}(q))^m)
    \]
    also exhibits the cyclic sieving phenomenon.
\end{theorem}

    This theorem serves as a general lifting theorem from which our main theorems follow as special cases.

\begin{proof}[Proof of \Cref{thm:mainA,thm:mainB}]
    \noindent \textbf{Rectangular Case.}
    For a rectangular partition $\lambda$, \Cref{thm:Rhoades} provides a CSP for $(\SYT(\lambda), \langle \pr \rangle, q^{-\kappa(\lambda)}f^\lambda(q))$. Applying \Cref{thm:unified} with $m=1$, $g=\pr$, and $\alpha = -\kappa(\lambda)$ yields the CSP for signed tableaux as stated in \Cref{thm:mainA}.
    
    \noindent \textbf{Non-Rectangular Case.}
    For an arbitrary partition $\lambda$, \Cref{thm:APRU} guarantees a CSP for $(\SYT(\lambda)^m, C_n, (f^\lambda(q))^m)$ whenever the evaluations are non-negative. Applying \Cref{thm:unified} with this base case (setting $\alpha=0$) directly implies \Cref{thm:mainB}.
\end{proof}

Before proving \Cref{thm:unified}, we start by evaluating the product term in \eqref{eq:factorisation} at a root of unity.

\begin{proposition} \label{prop:cyclotomic_identity}
    Let $\zeta$ be a primitive $s$-th root of unity. Then
    \begin{equation}\label{eq:cyclotomic_identity}
        \prod_{j=0}^{s-1} (1 + t \zeta^j) = 1 - (-t)^s.
    \end{equation}
\end{proposition}
\begin{proof}
    We start with the factorization of $x^s - 1 = \prod_{j=0}^{s-1} (x - \zeta^j)$.
    Substituting $x = -1/t$ gives
    \[
        (-1/t)^s - 1 = \prod_{j=0}^{s-1} (-1/t - \zeta^j) = (-1/t)^s \prod_{j=0}^{s-1} (1 + t\zeta^j).
    \]
    Multiplying by $(-t)^s$ yields the desired identity.
\end{proof}

\begin{proposition} \label{prop:product_eval}
    Let $\lambda \vdash n$ be a partition. Let $d \mid n$ and let $\xi$ be a primitive $n$-th root of unity. Define $s = n/d$. If $\BST(\lambda, s) \neq \emptyset$, then
    \[
    \prod_{\square \in \lambda} (1 + t (\xi^d)^{c(\square)}) = (1 - (-t)^s)^d.
    \]
\end{proposition}

\begin{proof}
    Let $\zeta = \xi^d$. Note that $\zeta$ is a primitive $s$-th root of unity. 
    
    As $\BST(\lambda, s) \neq \emptyset$, the exponents $c(\square)$ modulo $s$ are equidistributed, by \Cref{lem:content_dist}. Thus, the product factors into groups based on the residues modulo $s$:
    \[
    \prod_{\square \in \lambda} (1 + t \zeta^{c(\square)}) = \left( \prod_{j=0}^{s-1} (1 + t \zeta^j) \right)^{d}.
    \]
    By \eqref{eq:cyclotomic_identity}, this simplifies to $\left( 1 - (-t)^{s} \right)^d$.
    
\end{proof}

We can now prove the unified theorem:
\begin{proof}[Proof of \Cref{thm:unified}]
    We must show that for any $d \ge 0$, the evaluation of the polynomial at $q = \xi^d$ equals the number of fixed points of $g^d \times (\cyc^m)^d$ on $\SYT_{\pm k}(\lambda)^m$. Since both the number of fixed points of a cyclic group element $g^d$ and the evaluation of the polynomial at $\xi^d$ only depend on $\gcd(n,d)$, it is sufficient to prove the identity for all $d$ that divide $n$.
    
    Let $d$ be a divisor of $n$, and let $s=n/d$. The root of unity is $\zeta = \xi^d$, a primitive $s$-th root of unity.
    
    Let $\operatorname{Fix}(d)$ denote the number of fixed points of the action of $g^d \times (\cyc^m)^d$ on $\SYT_{\pm k}(\lambda)^m$. This is the product of the fixed point counts on the components.
    By assumption, the number of fixed points of $g^d$ on $\SYT(\lambda)^m$ is given by
    \[
        \left|\left(\SYT(\lambda)^m\right)^{g^d}\right| = \zeta^{\alpha} (f^\lambda(\zeta))^m.
    \]
    For the action of $(\cyc^m)^d$ on $(\binom{[n]}{k})^m$, the number of fixed points is $\left(\qbinom[q=\zeta]{n}{k}\right)^m$ by \Cref{thm:subsets,lem:product_csp}. Using the evaluation of the $q$-binomial coefficients from \eqref{eq:q_binom_eval} we have
    \[
        \operatorname{Fix}(d) = \begin{cases}
        \zeta^{\alpha} \left(f^\lambda(\zeta) \binom{d}{k/s} \right)^m &\text{if $s\mid k$}\\
        0 &\text{otherwise.}
        \end{cases}
    \]
    
    Now we evaluate the polynomial $P(q) = q^{\alpha}(q^{\gamma(n,k)} f^\lambda_{\pm k}(q))^m$ at $q = \zeta$.
    Using the factorization \eqref{eq:factorisation}, we have
    \[
        P(\zeta) = \zeta^{\alpha} \left( \zeta^{\gamma(n,k)} f^\lambda(\zeta) [t^k] \prod_{\square \in \lambda} (1 + t \zeta^{c(\square)}) \right)^m.
    \]
    If $\BST(\lambda, s) = \emptyset$, then by \eqref{eq:MN Rule}, $f^\lambda(\zeta) = 0$ and thus, $\operatorname{Fix}(d)=P(\zeta)=0$.

    If $\BST(\lambda, s) \neq \emptyset$, then by \Cref{prop:product_eval},
    \[
        [t^k]\prod_{\square \in \lambda} (1 + t \zeta^{c(\square)}) = [t^k] (1 - (-t)^s)^{d} = \begin{cases}
        \binom{d}{k/s} (-1)^{k/s+k}&\text{if $s\mid k$}\\
        0 &\text{otherwise.}
        \end{cases}
    \]
    Then, in the case where $k$ is not divisible by $s$ we directly get $P(\zeta) = \operatorname{Fix}(d) = 0$.
    In the other case, let $j=k/s$.
    It remains to check that the twist factor $\zeta^{\gamma(n,k)}$ accounts for the sign difference. Indeed, using $\xi^n=1$ and $\zeta = \xi^d$, we have $\zeta^{\gamma(n,k)} = \xi^{d(n-1)\binom{k}{2}} = \zeta^{-\binom{k}{2}}$.
    We check that $\zeta^{-\binom{k}{2}} = (-1)^{j+k}$.
    If $s$ is odd, then $s \mid \binom{k}{2}$ and $j+k$ is even, so both sides are $1$.
    If $s$ is even, then $k$ is even and $\zeta^{s/2}=-1$.
    So
    \[
    \zeta^{-\binom{k}{2}} = (-1)^{-j(k-1)} = (-1)^{j} = (-1)^{j+k}.
    \]

    Therefore, $P(\zeta) = \operatorname{Fix}(d)$.
\end{proof}

\section*{Acknowledgements}
The author was partially supported by Oliver Pechenik's Discovery Grant (RGPIN-2021-02391) and Launch Supplement (DGECR-2021-00010) from
the Natural Sciences and Engineering Research Council of Canada and was also partially supported by Olya Mandelshtam's Discovery Grant (RGPIN-2021-02568).

We are grateful for helpful conversations with Josh Swanson and Oliver Pechenik.

\bibliographystyle{alpha}
\bibliography{references}

@article{RSW2004,
  title={The cyclic sieving phenomenon},
  author={Reiner, V. and Stanton, D. and White, D.},
  journal={Journal of Combinatorial Theory, Series A},
  volume={108},
  number={1},
  pages={17--50},
  year={2004},
  publisher={Elsevier}
}

@article{Rhoades2010,
  title={Cyclic sieving, promotion, and representation theory},
  author={Rhoades, B.},
  journal={Journal of Combinatorial Theory, Series A},
  volume={117},
  number={1},
  pages={38--76},
  year={2010},
  publisher={Elsevier}
}

@ARTICLE{ArmonSwanson2025,
  title     = "Super major index and {T}hrall's problem",
  author    = "Armon, S. and Swanson, J.",
  journal   = "Algebraic Combinatorics",
  publisher = "Cellule MathDoc/Centre Mersenne",
  volume    =  8,
  number    =  3,
  pages     = "795--815",
  month     =  jun,
  year      =  2025,
  language  = "en"
}

@article{APRU2021,
  title={Skew characters and cyclic sieving},
  volume={9},
  DOI={10.1017/fms.2021.11},
  journal={Forum of Mathematics, Sigma},
  publisher={Cambridge University Press},
  author={Alexandersson, P. and Pfannerer, S. and Rubey, M. and Uhlin, J.},
  year={2021},
  pages={e41}
}

@book{James1984,
  doi = {10.1017/cbo9781107340732},
  year = {1984},
  month = dec,
  publisher = {Cambridge University Press},
  author = {James, G. and Kerber, A.},
  title = {The Representation Theory of the Symmetric Group}
}

@article{FrameRobinsonThrall1954,
  title={The hook graphs of the symmetric group},
  author={Frame, S. and Robinson, G. and Thrall, R.},
  journal={Canadian Journal of Mathematics},
  DOI={10.4153/CJM-1954-030-1},
  volume={6},
  pages={316--324},
  year={1954},
  publisher={Cambridge University Press}
}

@incollection{Sagan2011,
  title={The cyclic sieving phenomenon: a survey},
  author={Sagan, B.},
  booktitle={Surveys in Combinatorics 2011},
  pages={183--234},
  year={2011},
  publisher={Cambridge University Press}
}

@article{Stembridge1994,
  title={Some hidden relations involving the ten symmetry classes of plane partitions},
  author={Stembridge, J.},
  journal={Journal of Combinatorial Theory, Series A},
  volume={68},
  number={2},
  pages={372--409},
  year={1994},
  publisher={Elsevier}
}

@book{Stanley1999,
  place={Cambridge},
  series={Cambridge Studies in Advanced Mathematics},
  title={Enumerative Combinatorics},
  volume={2},
  DOI={10.1017/CBO9780511609589},
  publisher={Cambridge University Press},
  author={Stanley, R.},
  year={1999},
  collection={Cambridge Studies in Advanced Mathematics}
}

\end{document}